\documentclass[12 pt]{amsart}
\usepackage{amsmath}
\usepackage{amsfonts}
\usepackage{amssymb,multicol}
\usepackage{multicol} 
\usepackage[top=1.1in,bottom=1.2in,left=1.30in,right=1.30in]{geometry}
 
\newtheorem{theorem}{Theorem}
\theoremstyle{plain}

\newtheorem{lemma}{Lemma}

\numberwithin{equation}{section}
\numberwithin{lemma}{section}
\numberwithin{theorem}{section}
\numberwithin{corollary}{section}
\numberwithin{proposition}{section}

\usepackage{hyperref}  
\hypersetup{colorlinks=true,
linkcolor= blue,
filecolor=magenta,
citecolor= magenta,
urlcolor=cyan,}
\begin{document}
\title{ Lie Nilpotency Index of a Modular Group Algebra}
\author{Meena Sahai}
\address{Department of Mathematics and Astronomy, University of Lucknow,
Lucknow, U.P. 226007, India}
\email{meena\_sahai@hotmail.com}
\author{Bhagwat Sharan}
\address{Department of Mathematics and Astronomy, University of Lucknow,
Lucknow, U.P. 226007, India}
\email{bhagwat\_sharan@hotmail.com}
\begin{abstract}
 In this paper, we classify the modular group algebra $KG$ of a group $G$ over a field $K$ of characteristic $p>0$ having upper Lie nilpotency index $t^{L}(KG)= \vert G^{\prime}\vert - k(p-1) + 1$ for $k=14$ and $15$. Group algebras of upper Lie nilpotency index $\vert G^{\prime}\vert - k(p-1) + 1$ for $k\leq 13$, have already been characterized completely. 
\end{abstract}
\maketitle
\section{\textbf{Introduction and Preliminaries }}
Let $G$ be a group, not necessarily finite, and let $KG$ be its group algebra over a field $K$ of characteristic $p>0$. Let $KG^{[n]}$ and $KG^{(n)}$ denote the $n$th Lie power  and  upper Lie power of $KG$, respectively. The lower and the upper Lie nilpotency indices of $KG$  are  $t_{L}{(KG)} =  \min \{{n: KG^{[n]} = 0\}}$ and $t^{L}{(KG)} =  \min \{{n: KG^{(n)} = 0\}}$, respectively. In \cite[pp. 46, 48]{Pa}, for every $m\geq 1$, the $m$th Lie dimension subgroup of $G$ over $K$  is defined as
\begin{equation*}
	D_{(m),K}(G) = G\cap(1+KG^{(m)})=\underset{\left( i-1\right) p^{j}\geq m-1}{\Pi }\gamma_{i}\left( G\right)
	^{p^{j}}.
\end{equation*}

If $KG$ is Lie nilpotent such that $|G'|=p^{n}$, then according to Jennings' theory (see \cite{Sh1}), the upper Lie nilpotency index $t^{L}(KG)$ $=$ $2+$ $(p-1)\sum_{m\geq1}md_{(m+1)}$, where $p^{d_{(m)}} = | D_{(m),K}(G) : D_{(m+1),K}(G)|$ for every $m\geq 2$. It is clear that $\sum_{m\geq2} d_{(m)}=n$. A detailed study of Lie nilpotent group algebras and their unit groups is given in \cite{BK}.   In \cite{SB}, it is proved that  if $G$ is a  nilpotent group with $|G'|=p^{n}$, then  $p+1\leq t_{L}(KG)\leq t^{L}(KG)\leq |G'|+1$. A complete description of Lie nilpotent modular group algebras $KG$ of upper Lie nilpotency index $t^{L}(KG)\leq 9p-7$  is given in \cite{Sa,msbs1,msbs2,msbs3,Sh3}. On the other hand, group algebras with $t^{L}(KG)= \vert G^{\prime}\vert - k(p-1) + 1$ for $k\leq 13$ are classified in \cite{Bo,BJS,BSp,BS,msrs2,msrs1,SB}. In this paper, we classify the group algebra $KG$ with $t^{L}(KG)= \vert G^{\prime}\vert - k(p-1) + 1$ for $k=14$ and $15$. our terminology and notations  are same as in \cite{msbs1}. Throughout this paper, $S(n,m)$ denotes the group number $m$ of order $n$ from the Small Groups Library-GAP \cite{GAP}. We have freely used the following lemma in our work:

\begin{lemma} \cite{BP, Sh2, Sh3}{\label{thm1}}
 Let $K$ be a field of characteristic $p>0$ and let $G$ be a nilpotent group such that $ |G^{\prime}| =p^{n}$ and $exp(G^{\prime})= p ^{e}$. 
\begin{enumerate}
\item If  $d_{(m+1)}=0$ and $m$ is a power of $p$, then $d_{( s+1) }=0$ for all $s\geq $ $m$ $($i.e. $D_{( m+1),k}(G)=1)$.
\item If $d_{(m+1)}=0$ and $p^{e-1}$ divides $m$, then $D_{( m+1),k}(G)=1$.
\item If $p\geq 5$ and $t_{L}(KG) < p^{n} + 1$, then $t_{L}(KG)=t^{L}(KG)\leq p^{n-1} + 2p - 1$.
\item If $d_{(l+1)}=0$ for some $l < pm$, then $d_(pm + 1)\leq d_{(m+1)}$.
\item If $\ d_{(m+1)}=0$, then $d_{(s+1)}=0$ for all $s\geq m$ with $\upsilon _{p'}(s)\geq \upsilon_{p'}(m)$, where $\upsilon _{p'}(x)$ is the maximal divisor of $x$ which is relatively prime to $p$.
\end{enumerate}
\end{lemma}
\section {Main Results}
\begin{lemma} {\label{l1}}
	Let $G$ be a group and let $K$ be a field of characteristic $p>0$ such that $KG$ is Lie nilpotent. Then  $t^{L}(KG) = \vert G^{\prime}\vert - 14p + 15$ if and only if  $p=2$, $d_{(2)} = d_{(5)} = d_{(9)}=1$, $d_{(3)}=2$. 
\end{lemma} 
\begin{proof}
Let $t^{L}(KG) = \vert G^{\prime}\vert - 14p + 15$ and let $\vert G'\vert = p^{n}$. If $n=1$, then $G'$ is cyclic and hence by \cite[Theorem 1]{BSp}, $t^{L}(KG) = \vert G^{\prime}\vert  + 1$. If $p=2$, then $t^{L}(KG) = 2^{n} -13$ and so we must have $n\geq 4$. Now if $n=4$, then $3 = t^{L}(KG)\geq 6$. Let $n=5$. So $t^{L}(KG) = 19$,  $\sum\limits_{m=1}^{17}md_{(m+1)}=17$ and $d_{(k)}=0$ for all $k\geq 19$. Also $\sum\limits_{m\geq2}d_{(m)}=5$ and $d_{(2)}\neq 0$. If $d_{(3)}=0$, then $D_{(3), K}(G)=1$ and  $d_{(2)}=17$, which is not possible. Similarly if $d_{(5)}=0$, then $D_{(5), K}(G)=1$ and $d_{(2)} + 2d_{(3)} +3 d_{(4)}=17$. Since $d_{(3)}\neq 0$, so $d_{(4)}\leq 3$ and  $d_{(2)} + 2d_{(3)} +3 d_{(4)} < 17$. If $d_{(9)}=0$, then  $D_{(9), K}(G)=1$ and  $\sum\limits_{m=1}^{7}md_{(m+1)}=17$. Now it is easy to see that $d_{(2)}, d_{(3)}, d_{(5)}\notin \{ 3, 4, 5\} $.  If $d_{(5)}=2$, then the only possible solution is $d_{(2)}= d_{(3)}= d_{(7)}=1$. But then  $\upsilon_{2'}(6)=3=\upsilon_{2'}(3)$, which yields $d_{(7)}=0$, by Lemma \ref{thm1}. Let $d_{(5)}=1$. Then the possible solutions are $d_{(2)}= d_{(3)}= 1$, $ d_{(6)}=2$ and $d_{(2)}= d_{(3)} = d_{(4)} = d_{(8)} = 1  $. Again these cases are not possible by Lemma \ref{thm1}(v). Hence $d_{(9)}\neq 0$ and we have $d_{(2)}=  d_{(5)} = d_{(9)} = 1 $, $d_{(3)} =2$.
 
Let $p=2$ and $n\geq 6$. We claim that $d_{(2^{i}+1)}>0$ for $0\leq i\leq  n-2$. Let, if possible, there exist some $0\leq s\leq  n-2$ such that $d_{(2^{s}+1)}=0$, then $s\neq 0$ and by Lemma \ref{thm1}, $D_{(2^{s}+1), K}(G)=1$. Thus $d_{(r)}=0$ for all $r\geq 2^{s}+1$. Moreover, if $d_{(q+1)}\neq 0$, then $q < 2^{s}$. Therefore we get 
\begin{align*}
 t^{L}(KG)&= 2+ \sum\limits_{i=0}^{s-1} 2^{i}d_{(2^{i}+1)} + \sum\limits_{q\neq 2^{i}}^{}qd_{(q+1)}\\
&  = 2+  \sum\limits_{i=0}^{s-1} 2^{i} + \sum\limits_{i=0}^{s-1} 2^{i}(d_{(2^{i}+1)}-1) + \sum\limits_{q\neq 2^{i}}^{}qd_{(q+1)}\\
& < 2+\sum\limits_{i=0}^{s-1} 2^{i} + \left (\sum\limits_{i=0}^{s-1} (d_{(2^{i}+1)}-1) + \sum\limits_{q\neq 2^{i}}^{}d_{(q+1)}\right ) 2^{s}\\
& = 1 + 2^{s} + (n-s)2^{s}\\
&  < 1 + 2^{n-2} +  \left\lbrace n- (n-2)   \right\rbrace  2^{n-2}\\
&  = 1 +  3.2^{n-2} < 2^{n}-13.
\end{align*}
which is a contradiction to the assumption that $t^{L}(KG) = 2^{n}-13$. Hence $d_{(2^{i}+1)}>0$ for $0\leq i\leq  n-2$ and there exists $\alpha\geq2$ such that $d_{(\alpha +1)}=1$. Thus
\begin{align*}     
2^{n}-13 = t^{L}(KG) & = 2+ \sum\limits_{i=0}^{n-2} 2^{i} + \alpha d_{(\alpha +1)}\\
& = 1+ 2^{n-1} + \alpha.
\end{align*}  
So $\alpha = 2^{n-1}-14$,  $d_{(2^{i}+1)}= d_{(2^{n-1} -14 )+1}=1$ and $d_{(j)}=0$ where $0\leq i\leq  n-2$, $j\neq 2^{i}+1$, $j\neq 2^{n-1}- 13$ and $j>1$. Let $ m = 2^{n-2}-7$ and $l= \alpha = 2^{n-1}-14$. Since $n\geq 6$, $\upsilon _{2'}(l)= \upsilon_{2'}(m)$ and $d_{(m+1)}=0$, so by Lemma \ref{thm1}(v), $d_{(\alpha +1)}=0$, a contradiction.

Let $p=3$. Then $t^{L}(KG) = 3^{n}-27$. So we must have $n\geq4$. If there exists some $0\leq s\leq n-2$ such that $d_{(3^{s}+1)}=0$, then 
\begin{align*}
t^{L}(KG)&= 2+ 2\sum\limits_{i=0}^{s-1} 3^{i}d_{(3^{i}+1)} + 2\sum\limits_{q\neq 3^{i}}qd_{(q+1)}\\
&  = 1 +  5.3^{n-2} < 3^{n}-27.
\end{align*}
Thus $d_{(3^{i}+1)}>0$ for $0\leq i\leq n-2$ and hence $d_{(\alpha +1)}=1$ for $\alpha = 3^{n-1}-14$. For $n=4$, $\upsilon_{3'}(\alpha)=13\geq\upsilon_{3'}(4)$ and so $d_{(\alpha +1)}=0$ by Lemma \ref{thm1}(v). For $n\geq5$, since $exp(G')\leq 3^{n-1}$,   
and since $d_{(2.3^{n-2}+1)}=0$, so by Lemma \ref{thm1}(ii), $D_{(2.3^{n-2}+1),K}(G)=1$. But then $d_{(r)}=0$ for all $r\geq 2.3^{n-2}+1$ and so $d_{(3^{n-1}-13) = 0}$.

If $p\in \{5, 7, 11, 13 \}$, then by Lemma \ref{thm1}(iii), $n\leq 2$ and so $t^{L}(KG)<0$. If $p\geq 17$, then $n=1$, by Lemma \ref{thm1}(iii).   
\end{proof}
\begin{theorem}
Let $K$ be a field of characteristic $p>0$ and let $G$ be a nilpotent group such that $\vert G'\vert = p^{n}$. Then $t^{L}(KG)= |G'|- 14p+ 15$ if and only if one of the following conditions holds:
\begin{enumerate}
\item $G'\cong C_{16}\times C_{2}$, $G'^{2}\subseteq\gamma_{3}(G)\cong C_{8}\times C_{2}$ and $\gamma_{4}(G)\subseteq\gamma_{3}(G)^{2}$;
\item \begin{enumerate}
\item $G'\cong C_{16}\times C_{2}$, $\gamma_{3}(G)\cong C_{8}$ and $\gamma_{4}(G)\subseteq G'^{4}=\gamma_{3}(G)^{2}=G'^{2}\cap\gamma_{3}(G)\cong C_{4}$;
			
\item $G'\cong C_{16}\times C_{2}$, $\gamma_{3}(G)\cong C_{4}\times C_{2}$, $G'^{2}\cap\gamma_{3}(G)\cong C_{4}$ and $\gamma_{4}(G)$, $\gamma_{3}(G)^{2}\subseteq G'^{4}$;
\end{enumerate}
\item $G'\cong C_{16}\times C_{2}$, $|\gamma_{3}(G)|=4$, $\vert G'^{2}\cap\gamma_{3}(G)\vert= 2$ and $\gamma_{4}(G)$, $\gamma_{3}(G)^{2}\subseteq G'^{4}$;
		
\item $G'\cong C_{16}\times C_{2}$, $\gamma_{3}(G)\cong C_{2}$ and $G'^{2}\cap\gamma_{3}(G)= 1$; 
		
\item $G'\cong C_{8}\times C_{4}$, $G'^{2}\subseteq\gamma_{3}(G)\cong C_{8}\times C_{2}$ and $\gamma_{4}(G)\subseteq\gamma_{3}(G)^{2}$;
		
\item $G'\cong C_{8}\times C_{4}$, $\gamma_{3}(G)\cong C_{8}$, $\vert G'^{2}\cap\gamma_{3}(G)\vert= 2$ and $\gamma_{4}(G)$, $G'^{4}\subseteq\gamma_{3}(G)^{2}$;
		
\item $G'\cong C_{8}\times C_{2}\times C_{2}$, $G'^{2}\subseteq\gamma_{3}(G)\cong C_{8}\times C_{2}$ and $\gamma_{4}(G)\subseteq\gamma_{3}(G)^{2}$;
		
\item $G'\cong C_{8}\times C_{2}\times C_{2}$, $\gamma_{3}(G)\cong C_{8}$, $\vert G'^{2}\cap\gamma_{3}(G)\vert= 2$ and $\gamma_{4}(G)$, $G'^{4}\subseteq\gamma_{3}(G)^{2}$;
		
\item $G'\cong S(32,4)$ or $S(32,5)$ or $S(32,12)$,  $G'^{2}\subseteq\gamma_{3}(G)\cong C_{8}\times C_{2}$, $\gamma_{4}(G)= \gamma_{3}(G)^{2}$ and $\gamma_{5}(G)\cong C_{2}$;   
		
\item $G'\cong S(32,4)$, $\gamma_{3}(G)\cong C_{8}$, $G'^{4}\subseteq G'^{2}\cap\gamma_{3}(G)= \gamma_{4}(G)= \gamma_{3}(G)^{2}$ and  $\gamma_{5}(G)\cong C_{2}$; 

\item $G'\cong S(32,37)$  and  $G'^{2}\subseteq\gamma_{3}(G)\cong C_{8}\times C_{2}$, $\gamma_{4}(G)= \gamma_{3}(G)^{2}$, $\gamma_{5}(G)\cong C_{2}$; 
		
\item $G'\cong S(32,17)$, $\gamma_{5}(G)\cong C_{2}$  and  $G'^{2}\subseteq\gamma_{3}(G)\cong C_{8}\times C_{2}$, $\gamma_{4}(G)= \gamma_{3}(G)^{2}$  or $\gamma_{3}(G)\cong C_{8}$, $ G'^{2}\cap\gamma_{3}(G)= \gamma_{4}(G)= \gamma_{3}(G)^{2}= G'^{4}$ or $\gamma_{3}(G)\cong C_{4}\times C_{2}$, $\gamma_{3}(G)^{2}\subseteq G'^{2}\cap\gamma_{3}(G)= \gamma_{4}(G)=G'^{4}\cong C_{4}$;
		
\item $G'\cong S(32,38)$, $G'^{2}\subseteq\gamma_{3}(G)\cong C_{8}\times C_{2}$, $\gamma_{4}(G)= \gamma_{3}(G)^{2}$ and $\gamma_{5}(G)\cong C_{2}$. 
\end{enumerate}	
\end{theorem}	

\begin{proof}
Let $t^{L}(KG)= |G'|- 14p+ 15$. Then by Lemma \ref*{l1}, $p=2$ and $d_{(2)} = d_{(5)} = d_{(9)}=1$, $d_{(3)}=2$. Therefore $|G'|= 2^{5}$, $|D_{(3),K}(G)|=2^{4}$, $|D_{(4),K}(G)|=|D_{(5),K}(G)|=2^{2}$, $|D_{(6),K}(G)|= |D_{(7),K}(G)|= |D_{(8),K}(G)|$ $ = |D_{(9),K}(G)|=2$ and $D_{(10),K}(G)= 1$. Since $|D_{(4),K}(G)|= 4$, so $|\gamma_{4}(G)|\leq 4$ and  $\gamma_{6}(G)=1$. Thus  $G'$ is nilpotent of class at most $2$, $\gamma_{3}(G)$ is abelian and $\gamma_{4}(G)\subseteq\zeta(G')$. Also $\vert\gamma_{5}(G)\vert\leq 2$, so $\gamma_{4}(G)^{2}\subseteq\gamma_{3}(G)^{4}G'^{8}\cong C_{2}$ and $G'^{4}\neq 1$. Hence exponent of $\gamma_{3}(G)^{2}G'^{4}$ is not equal to $2$ and $\gamma_{4}(G)\subseteq\gamma_{3}(G)^{2}G'^{4}\cong C_{4}$.
	
Let $G'$ be an abelian group. Then  $G'\cong C_{16}\times C_{2}$ or $C_{8}\times C_{4}$ or $ C_{8}\times C_{2}\times C_{2}$. If $G'\cong C_{16}\times C_{2}$, then $G'^{2}\subseteq\gamma_{3}(G)\cong C_{8}\times C_{2}$, $\gamma_{4}(G)\subseteq\gamma_{3}(G)^{2}$ or $\gamma_{3}(G)\cong C_{8}$, $\gamma_{4}(G)\subseteq G'^{4}=\gamma_{3}(G)^{2} = G'^{2}\cap\gamma_{3}(G)\cong C_{4}$  or  $\gamma_{3}(G)\cong C_{4}\times C_{2}$, $G'^{2}\cap\gamma_{3}(G)\cong C_{4}$, $\gamma_{3}(G)^{2}\subseteq G'^{4}$, $\gamma_{4}(G)\subseteq G'^{4}$ or  $|\gamma_{3}(G)|=4$, $G'^{2}\cap\gamma_{3}(G)\cong C_{2}$, $\gamma_{3}(G)^{2}\subseteq G'^{4}$, $\gamma_{4}(G)\subseteq G'^{4}$ or $|\gamma_{3}(G)|=2$, $G'^{2}\cap\gamma_{3}(G)=1$. If $G'^{8}=1$, then $\gamma_{3}(G)^{4}\cong C_{2}$, so $|\gamma_{3}(G)|=8$ or $16$. Thus $G'^{2}\subseteq\gamma_{3}(G)\cong C_{8}\times C_{2}$, $\gamma_{4}(G)\subseteq\gamma_{3}(G)^{2}$ or $\gamma_{3}(G)\cong C_{8}$, $G'^{2}\cap\gamma_{3}(G)\cong C_{2}$, $G'^{4}\subseteq\gamma_{3}(G)^{2}$, $\gamma_{4}(G)\subseteq\gamma_{3}(G)^{2}$.
	
Let $G'$ be a nonabelian group. If $\gamma_{5}(G)\cong C_{2}$, then $\gamma_{4}(G)=\gamma_{3}(G)^{2}G'^{4}$ $\cong C_{4}$. Thus $|\zeta(G')|=4 $ or $8$. If $|\zeta(G')|=4$, then $\gamma_{4}(G)=\zeta(G')\cong C_{4}$ and we see from Table \ref{tab:table1} that $cl(G') = 3$. Thus $\zeta(G')\cong C_{8}$ or $C_{4}\times C_{2}$. If $\zeta(G')\cong C_{8}$, then from Table \ref{tab:table1}, the possible choices for $G'$ are $S(32,17)$ or $S(32,38)$. If $G'\cong S(32,17)$, then $G'^{2}= \zeta(G')\cong C_{8}$ and so $G'^{2}\subseteq\gamma_{3}(G)\cong C_{8}\times C_{2}$, $\gamma_{4}(G)= \gamma_{3}(G)^{2}$ or $\gamma_{3}(G)\cong C_{8}$, $G'^{2}\cap\gamma_{3}(G)= \gamma_{4}(G)= G'^{4}=\gamma_{3}(G)^{2}\cong C_{4}$ or $\gamma_{3}(G)\cong C_{4}\times C_{2}$, $G'^{2}\cap\gamma_{3}(G)= \gamma_{4}(G)=G'^{4}\cong C_{4}$, $\gamma_{3}(G)^{2}\subseteq G'^{4}$. If $G'\cong S(32,38)$, then $G'^{2}\cong C_{4}$ and so  $G'^{2}\subseteq\gamma_{3}(G)\cong C_{8}\times C_{2}$, $\gamma_{4}(G)= \gamma_{3}(G)^{2}$. If $\zeta(G')\cong C_{4}\times C_{2}$, then the possible  choices for $G'$  are  $S(32,4)$ or $S(32,5)$ or $S(32,12)$ or $S(32,37)$. If $G'\cong S(32,4)$ or $S(32,5)$ or $S(32,12)$, then $G'^{2}\cong C_{4}\times C_{2}$  so either  $G'^{2}\subseteq\gamma_{3}(G)\cong C_{8}\times C_{2}$, $\gamma_{4}(G)= \gamma_{3}(G)^{2}$ or $\gamma_{3}(G)\cong C_{8}$, $\gamma_{4}(G)= G'^{2}\cap\gamma_{3}(G)\cong C_{4}$, $G'^{4}\subseteq\gamma_{4}(G)=\gamma_{3}(G)^{2}$. 
But $S(32,5)$ or $S(32,12)$ do not contain a normal subgroup of the type $C_{8}$.  If $G'\cong S(32,37)$, then $G'^{2}\cong C_{4}$ and so $G'^{2}\subseteq\gamma_{3}(G)\cong C_{8}\times C_{2}$, $\gamma_{4}(G)= \gamma_{3}(G)^{2}$.

If $\gamma_{5}(G)=1$, then $\gamma_{3}(G)\subseteq\zeta(G')$ and so $|\gamma_{3}(G)|\neq 16$. Also $G'^{4}\neq1$. Thus $|\gamma_{4}(G)|\geq2$ and $|\gamma_{3}(G)|= 4$ or $8$. If  $\gamma_{3}^{4}(G)=1$, then  $G'^{8}\cong C_{2}$. Therefore $G'\cong S(32,17)$. But then $\gamma_{3}(G)\subseteq\zeta(G')= G'^{2}$ and $|D_{(3),K}(G)|=8$. So $\gamma_{3}(G)=\zeta(G')\cong C_{8}$ and the possible choices for $G'$ are $S(32,17)$ or $S(32,38)$. If $G'\cong S(32,17)$ or $S(32,38)$ , then $G'^{2}\subseteq\zeta(G') \cong C_{8}$ and so $|D_{(3),K}(G)|\neq16$.\\
Conversely, if conditions (i) to (xiii) holds, then  $d_{(2)} = d_{(5)} = d_{(9)}=1$, $d_{(3)}=2$. So $n=5$ and $t^{L}(KG)=19= |G'|- 14p+ 15$.      

\end{proof}	

\begin{lemma}{\label{l2}}
Let $G$ be a group and let $K$ be a field of characteristic $p>0$ such that $KG$ Lie nilpotent. Then  $t^{L}(KG) = \vert G^{\prime}\vert - 15p + 16$ if and only if one of the following conditions holds :
\begin{enumerate}
\item   $p=2$, $n=5$ and $d_{(2)} = d_{(3)} = d_{(4)} = d_{(5)}= d_{(7)}=1$;
\item   $p=2$, $n=5$ and $d_{(2)}=2$, $ d_{(3)} = d_{(5)}= d_{(9)}=1$; 
\item $p=17$, $n=2$ and $d_{(2)} = d_{(3)} =1$. 
\end{enumerate}
\end{lemma}
\begin{proof}
Let $t^{L}(KG) = \vert G^{\prime}\vert - 15p + 16$ and let $\vert G^{\prime}\vert = p^{n}$. As in Lemma \ref{l1}, we can not have $n=1$. It is easy to see that if $p=2$, then $n\geq 5$, if $p=3$, then $n\geq 4$, if $p\in \{5, 7, 11, 13\}$, then $n\geq 3$ and if $p\geq 17$, then $n\geq 2$. 
 
 If $p=2$ and $n=5$, then we have $\sum\limits_{m=1}^{16}md_{(m+1)}=16$ and $d_{(k)}=0$ for all $k\geq 17$. Also $\sum_{m\geq 2} d_{(m)}=5$ and $d_{(2)}$, $d_{(3)}$, $d_{(5)}\neq 0$. If $d_{(9)}=0$, then $\sum\limits_{m=1}^{7}md_{(m+1)}=16$. Clearly $d_{(2)}, d_{(3)}, d_{(5)} \notin\{ 3, 4, 5 \}$. If $d_{(5)}=2$, then  $d_{(2)} = d_{(3)} = d_{(6)}=1$. But then $2= d_{(2.2+1)}\leq d_{(2+1)}=1$ by Lemma \ref{thm1}(iv). So $d_{(5)}=1$, and either $d_{(2)}= d_{(8)}= 1$, $ d_{(3)}=2$ or $d_{(2)}= d_{(3)} = d_{(4)} = d_{(7)} = 1$. If  $d_{(2)}= d_{(5)}= d_{(8)}= 1$, $ d_{(3)}=2$, then  $\upsilon _{2'}(7)\geq \upsilon_{2'}(5)$ and  $d_{(8)}=0$ by Lemma \ref{thm1}(v), so this is not possible. Clearly, if $d_{(9)}\neq 0$, then $d_{(2)}=2$, $d_{(3)}= d_{(5)}= d_{(9)} =1$.
 
Let $p=2$ and $n\geq 6$, then proceeding in the same way as we have done in Lemma \ref*{l1}, we obtain  $d_{(2^{i}+1)}= d_{(2^{n-1} -15 )+1}=1$ and $d_{(j)}=0$ where $0\leq i\leq  n-2$, $j\neq 2^{i}+1$, $j\neq 2^{n-1}- 14$ and $j>1$. Let $ m = 2^{n-2}+1$ and $l=  2^{n-1}-15$. Since $n\geq 6$, $\upsilon_{2'}(l)\geq \upsilon_{2'}(m)$ and by Lemma \ref{thm1}(v) we get $d_{(2^{n-1}-14)}=0$,  which is a contradiction.

If $p=3$ and $n\geq 4$, then similarly  we get $d_{(3^{i}+1)}= d_{(3^{n-1} -15 )+1}=1$ and $d_{(j)}=0$ where $0\leq i\leq n-2$, $j\neq 3^{i}+1$, $j\neq 3^{n-1}- 14$ and $j>1$. Let $ m = 3^{n-2}-5$ and $l=  3^{n-1}-15$. Since $n\geq 4$, $\upsilon_{3'}(l)= \upsilon_{3'}(m)$ and so $d_{(l+1)}=0$,  a contradiction.

Let $p\in \{5, 7, 11, 13 \}$. Then by Lemma \ref{thm1}(iii), $n=2$ which is not possible. Also if $p=17$, then $n=2$ and $d_{(2)}= d_{(3)}=1$. We can not have $p\geq 19$  by Lemma \ref{thm1}(iii).  
\end{proof}

\begin{lemma}{\label{l3}} 
Let $K$ be a field with characteristic $p=2$ and let $G$ be a nilpotent group such that $\vert G'\vert = 2^{n}$. Then $t^{L}(KG)= 2^{n}-14$ if and only if one of the following conditions holds:

\begin{enumerate}
\item $G'\cong C_{16}\times C_{2}$, $\gamma_{3}(G)\subseteq G'^{2}$ and $\gamma_{4}(G)\subseteq G'^{4}$;
\item $G'\cong C_{8}\times C_{2}\times C_{2}$, $G'^{2}\subseteq \gamma_{3}(G)\cong C_{8}$ and $\gamma_{4}(G)\subseteq \gamma_{3}(G)^{2}$;
\item \begin{enumerate} \item $G'\cong S(32,17)$ or $S(32,38)$, $G'^{2}\subseteq \gamma_{3}(G)\cong C_{8}$, $\gamma_{4}(G)\subseteq \gamma_{3}(G)^{2}$ and $\gamma_{5}(G)=1$;		
	
\item $G'\cong S(32,17)$ or $S(32,38)$, $G'^{2}\subseteq \gamma_{3}(G)\cong C_{8}$, $\gamma_{4}(G)=\gamma_{3}(G)^{2}$ and $\gamma_{5}(G)\cong C_{2}$;		
\end{enumerate}
\item $G'\cong S(32,37)$, $G'^{2}\subseteq \gamma_{3}(G)\cong C_{8}$, $\gamma_{4}(G)=\gamma_{3}(G)^{2}$ and $\gamma_{5}(G)\cong C_{2}$;		

\item $G'\cong S(32,17)$, $\gamma_{3}(G)\subseteq G'^{2}\cong C_{8}$, $\gamma_{3}(G)\cong C_{4}$;		

\item $G'\cong C_{8}\times C_{2}\times C_{2}$ or $C_{4}\times C_{4}\times C_{2}$ or  $C_{4}\times C_{2}\times C_{2}\times C_{2}$, $G'^{2}\subseteq\gamma_{3}(G)\cong C_{4}\times C_{2}\times C_{2}$, $\gamma_{4}(G)\cong C_{4}\times C_{2}$ and $\gamma_{6}(G)\subseteq\gamma_{3}(G)^{2}= \gamma_{4}(G)^{2}\subseteq\gamma_{5}(G)\cong C_{2}\times C_{2}$;	 
\item $G'\cong S(32,5)$ or $S(32,24)$ or $S(32,25)$ or $S(32,37)$ or $S(32,48)$, $G'^{2}\subseteq\gamma_{3}(G)\cong C_{4}\times C_{2}\times C_{2}$, $\gamma_{4}(G)\cong C_{4}\times C_{2}$, $\gamma_{3}(G)^{2}\subseteq\gamma_{5}(G)\cong C_{2}\times C_{2}$ and $\gamma_{6}(G)=1$;

\item $G'\cong S(32,28)$ to $S(32,30)$,  $G'^{2} \subseteq \gamma_{3}(G)\cong C_{4}\times C_{2}\times C_{2}$, $\gamma_{4}(G)\cong C_{4}\times C_{2}$ and $\gamma_{6}(G)=\gamma_{3}(G)^{2}\subseteq\gamma_{5}(G)\cong C_{2}\times C_{2}$; 	 

\item $G'\cong S(32,2)$ or  $S(32,5)$  or $S(32,22)$ to $S(32,25)$ or $S(32,37)$ or $S(32,46)$ to $S(32,48)$,  $G'^{2} \subseteq \gamma_{3}(G)\cong C_{4}\times C_{2}\times C_{2}$, $\gamma_{4}(G)\cong C_{4}\times C_{2}$ and $\gamma_{6}(G)=\gamma_{3}(G)^{2}\subseteq\gamma_{5}(G)\cong C_{2}\times C_{2}$; 	 

\item $G'\cong C_{8}\times C_{2}\times C_{2}$ or $C_{4}\times C_{4}\times C_{2}$ or  $C_{4}\times C_{2}\times C_{2}\times C_{2}$, $G'^{2}\subseteq\gamma_{3}(G)\cong C_{4}\times C_{2}\times C_{2}$, $\gamma_{4}(G)\cong C_{4}\times C_{2}$, $\gamma_{3}(G)^{2}\subseteq\gamma_{4}(G)$, $\gamma_{5}(G)\cong C_{2}$ and	 $\gamma_{5}(G)\cap\gamma_{3}(G)^{2}=1$; 

\item $G'\cong S(32,5)$ or $S(32,24)$, $S(32,25)$ or $S(32,37)$ or $S(32,48)$, $G'^{2}\subseteq\gamma_{3}(G)\cong C_{4}\times C_{2}\times C_{2}$, $\gamma_{4}(G)\cong C_{4}\times C_{2}$, $\gamma_{3}(G)^{2}\subseteq\gamma_{4}(G)$, $\gamma_{5}(G)\cong C_{2}$ and	 $\gamma_{5}(G)\cap\gamma_{3}(G)^{2}=1$;

\item $G'\cong C_{8}\times C_{4}$ or $C_{4}\times C_{4}\times C_{2}$, $G'^{2}\subseteq\gamma_{3}(G)\cong C_{4}\times C_{4}$, $\gamma_{5}(G)\subseteq\gamma_{3}(G)^{2}\subseteq\gamma_{4}(G)\cong C_{4}\times C_{2}$;

\item $G'\cong S(32,4)$  or $S(32,12)$ or $S(32,24)$ to $S(32,26)$, $G'^{2}\subseteq\gamma_{3}(G)\cong C_{4}\times C_{4}$, $\gamma_{4}(G)\cong C_{4}\times C_{2}$, $\gamma_{5}(G)=\gamma_{3}(G)^{2}$ and $\gamma_{6}(G)=1$;

\item $G'\cong S(32,4)$ or $S(32,12)$ or $S(32,24)$ to $S(32,26)$ or $S(32,31)$ to $S(32,35)$, $G'^{2}\subseteq\gamma_{3}(G)\cong C_{4}\times C_{4}$, $\gamma_{4}(G)\cong C_{4}\times C_{2}$, $\gamma_{5}(G)=\gamma_{3}(G)^{2}$ and $\gamma_{6}(G)\cong C_{2}$;

\item $G'\cong S(32,4)$ or $S(32,12)$ or $S(32,24)$ to $S(32,26)$, $G'^{2}\subseteq\gamma_{3}(G)\cong C_{4}\times C_{4}$, $\gamma_{5}(G)\subseteq\gamma_{3}(G)^{2}\subseteq\gamma_{4}(G)\cong C_{4}\times C_{2}$ and $\gamma_{5}(G)\cong C_{2}$;

\item $G'\cong C_{8}\times C_{4}$ or $C_{4}\times C_{4}\times C_{2}$, $G'^{2}\subseteq\gamma_{3}(G)\cong C_{4}\times C_{4}$, $\gamma_{4}(G)\cong C_{4}$, $|\gamma_{3}(G)^{2}\cap \gamma_{4}(G)|=2$ and $\gamma_{5}(G)\subseteq\gamma_{3}(G)^{2}$;

\item $G'\cong S(32,4)$ or $S(32,12)$ or $S(32,24)$ to $S(32,26)$, $G'^{2}\subseteq\gamma_{3}(G)\cong C_{4}\times C_{4}$, $\gamma_{4}(G)\cong C_{4}$, $|\gamma_{3}(G)^{2}\cap \gamma_{4}(G)|=2$, $\gamma_{5}(G)\subseteq\gamma_{3}(G)^{2}$ and $\gamma_{5}(G)\cong C_{2}$;

\item $G'\cong S(32,4)$ or $S(32,5)$ or $S(32,12)$ or $S(32,17)$, $\gamma_{3}(G)\cong C_{4}\times C_{2}$, $|G'^{2}\cap \gamma_{3}(G)|=4$, $\gamma_{4}(G)\cong C_{4}$, $|\gamma_{3}(G)^{2}G^{\prime 4}\cap \gamma_{4}(G)|=2$,
$\gamma_{5}(G)\cong C_{2}$ and $\gamma_{5}(G)\subseteq \gamma_{3}(G)^{2}G^{\prime 4}\cong C_{2}\times C_{2}$;

\item $G'\cong C_{8}\times C_{2}\times C_{2}$ or $C_{4}\times C_{4}\times C_{2}$, $\gamma_{3}(G)\cong C_{4}\times C_{2}$, $|G'^{2}\cap\gamma_{3}(G)|=2$, $\gamma_{4}(G)\cong C_{4}$, $|\gamma_{3}(G)^{2}G^{\prime 4}\cap\gamma_{4}(G)|=2$
and $\gamma_{5}(G)\subseteq \gamma_{3}(G)^{2}G^{\prime 4}\cong C_{2}\times C_{2}$;

\item $G'\cong S(32,24)$ or $S(32,25)$ or $S(32,26)$ or $S(32,37)$ or $S(32,38)$, $\gamma_{3}(G)\cong C_{4}\times C_{2}$, $|G'^{2}\cap \gamma_{3}(G)|=2$, $\gamma_{4}(G)\cong C_{4}$, $|\gamma_{3}(G)^{2}G^{\prime 4}\cap \gamma_{4}(G)|=2$,
$\gamma_{5}(G)\cong C_{2}$ and $\gamma_{5}(G)\subseteq \gamma_{3}(G)^{2}G^{\prime 4}\cong C_{2}\times C_{2}$.

\end{enumerate}
\end{lemma}

\begin{proof}
Let $t^{L}(KG)= 2^{n}-14$, then by Lemma \ref{l2}, we have either  $d_{(2)} = d_{(3)}= d_{(4)} = d_{(5)} = d_{(7)} = 1$ or $d_{(2)}=2$, $d_{(3)} = d_{(5)} = d_{(9)}=1$. We consider these two cases separately.\\
{\bf Case (I).} $d_{(2)}=2$, $d_{(3)} = d_{(5)} = d_{(9)}=1$.
Thus $|G'|= 2^{5}$, $|D_{(3),K}(G)|=2^{3}$, $|D_{(4),K}(G)|=|D_{(5),K}(G)|=2^{2}$, $D_{(6),K}(G)=D_{(7),K}(G)=D_{(8),K}(G)= D_{(9),K}(G) \cong C_2$ and $D_{(10),K}(G)= 1$. Since $|D_{(4),K}(G)|$ $= 4$, so $\gamma_{6}(G)=1$ and nilpotency class of $G'$ is at most $2$. Also $|\gamma_{5}(G)^{2}\gamma_{3}(G)^{4}G'^{8}|= 2$ leads to $|\gamma_{3}(G)^{4}G'^{8}|=2$.
Hence $G'^{4}\neq1$, $exp(\gamma_{3}(G)^{2}G'^{4})= 4$ and $\gamma_{4}(G)\subseteq\gamma_{3}(G)^{2}G'^{4}\cong C_{4}$. Let $G'$ be an abelian group. Then $G'\cong C_{16}\times C_{2}$ or $C_{8}\times C_{4}$ or $ C_{8}\times C_{2}\times C_{2}$. If $G'\cong C_{16}\times C_{2}$, then $\gamma_{3}(G)\subseteq G'^{2}$ and $\gamma_{4}(G)\subseteq G'^{4}$. If $G'\cong C_{8}\times C_{4}$, then again  $\gamma_{3}(G)\subseteq G'^{2}$. But then $\gamma_{3}(G)^{4}G'^{8} = G'^{8}=1$. If $G'\cong C_{8}\times C_{2}\times C_{2}$, then $|\gamma_{3}(G)^{4}G'^{8}|=2$ yields that
$G'^{2}\subseteq\gamma_{3}(G)\cong C_{8}$ and $\gamma_{4}(G)\subseteq\gamma_{3}(G)^{2}$.

Let $G'$ be a nonabelian group. Then $\gamma_{4}(G)\neq1$. If $|\gamma_{3}(G)|=8$, then $G'^{2}\subseteq\gamma_{3}(G)\cong C_{8}$. Thus if $\gamma_{5}(G)=1$, then $\gamma_{3}(G)=\zeta(G')$ and so possible choices for  $G'$ are $S(32,17)$ or $S(32,38)$. On the other hand, if $|\gamma_{5}(G)|=2$, then $G'^{4}\gamma_{3}(G)^{2}=\gamma_{4}(G)\subseteq\zeta(G')$ and hence $\zeta(G')\cong C_{4}$ or $C_{4}\times C_{2}$ or $C_{8}$. From Table \ref{tab:table1}, we see that we can not have $\zeta(G')\cong C_{4}$. Thus the possible choices for $G'$ are  $S(32,4)$ or $S(32,5)$ or $S(32,12)$ or $S(32,17)$ or $S(32,37)$ and $S(32,38)$. But if $G'\cong S(32,4)$ or $S(32,5)$ or $S(32,12)$, then $G'^{2}\cong C_{4}\times C_{2}$ whereas $\gamma_{3}(G)\cong C_{8}$. If $|\gamma_{3}(G)|= 4$, then $G'^{8}\cong C_{2}$, so $exp (G')=16$, $G'\cong S(32,17)$, $\gamma_{3}(G)\subseteq G'^{2}=\zeta(G')\cong C_{8}$ and 
$\gamma_{4}(G)\subseteq G'^{4}$.\\
{\bf Case (II).} $d_{(2)} = d_{(3)}= d_{(4)} = d_{(5)} = d_{(7)}=1$. Thus $|G'|= 2^{5}$, $|D_{(3),K}(G)|=2^{4}$, $|D_{(4),K}(G)|= 2^{3}$, $|D_{(5),K}(G)|=2^{2}$, $D_{(6),K}(G)= D_{(7),K}(G)\cong C_{2}$ and $D_{(8),K}(G)=\gamma_{8}(G)\gamma_{5}(G)^{2}\gamma_{3}(G)^{4}G^{\prime 8}=1$. Hence $|\gamma_{5}(G)|\leq4$, $\gamma_{6}(G)\subseteq\gamma_{4}(G)^{2}\cong C_{2}$, $\gamma_{7}(G)=1$ and  $|\gamma_{3}(G)^{2}G^{\prime 4}|=2$ or $4$.\\
{\bf{(a)}}  $|\gamma_{3}(G)^{2}G^{\prime 4}|=2$. Clearly $\gamma_{4}(G)^{2}= \gamma_{3}(G)^{2}=\gamma_{3}(G)^{2}G^{\prime 4}\cong C_{2}$.

If $|\gamma_{5}(G)|=4$, then $\gamma_{3}(G)^{2}G^{\prime 4}\subseteq \gamma_{5}(G)\cong C_{2}\times C_{2}$, $\gamma_{4}(G)\cong C_{4}\times C_{2}$ and  $G'^{2}\subseteq \gamma_{3}(G)$. Since $\gamma_{7}(G)=1$, we have $\gamma_{4}(G)\subseteq\zeta(\gamma_{3}(G))$ and hence $\gamma_{3}(G)$ must be abelian and isomorphic to $C_{4}\times C_{2}\times C_{2}$. Thus if $G'$ is abelian, then $G'\cong C_{8}\times C_{2}\times C_{2}$ or $C_{4}\times C_{4}\times C_{2}$ or $C_{4}\times C_{2}\times C_{2}\times C_{2}$. Now let $G'$ be nonabelian. If $\gamma_{6}(G)=1$, then $cl(G')=2$ and $\gamma_{4}(G)=\zeta(G')\cong C_{4}\times C_{2}$. 
So possible choices for $G'$ are $S(32,4)$ or $S(32,5)$ or $S(32,12)$ or  $S(32,24)$ or $S(32,25)$ or $S(32,26)$ or $S(32,37)$ or $S(32,48)$. Since $S(32,4)$ or $S(32,12)$ or $S(32,26)$ do not contain a normal subgroup of the type $C_{4}\times C_{2}\times C_{2}$, so these groups are not possible. Hence $G'\cong S(32,5)$ or  $S(32,24)$ or $S(32,25)$  or $S(32,37)$ or $S(32,48)$. If $\gamma_{6}(G)\cong C_{2}$, then $\gamma_{5}(G)\subseteq\zeta(G')$. Thus $|\zeta(G')|= 4$ or $8$. If $|\zeta(G')|= 4$, then $\gamma_{5}(G)=\zeta(G')\cong C_{2}\times C_{2}$ and hence possible choices for $G'$ are $S(32,9)$ or $S(32,10)$ or $S(32,13)$ or $S(32,14)$ or $S(32,27)$ to  $S(32,35)$ or $S(32,39)$ to $S(32,41)$. But the groups $S(32,9)$ or $S(32,10)$ or $S(32,13)$ or $S(32,14)$ or $S(32,27)$ or $S(32,31)$ to  $S(32,35)$ or $S(32,39)$ to $S(32,41)$, do not contain a normal subgroup of the type $C_{4}\times C_{2}\times C_{2}$. Hence $G'\cong S(32,28)$ to  $S(32,30)$. If $|\zeta(G')|= 8$, then either $\zeta(G')\cong C_{2}\times C_{2}\times C_{2}$ or $C_{4}\times C_{2}$. So possible choices for $G'$ are $S(32,2)$ or $S(32,4)$ or $S(32,5)$ or $S(32,12)$ or $S(32,22)$ to $S(32,26)$ or $S(32,37)$ or $S(32,46)$ to $S(32,48)$. But a normal subgroup of the type $C_{4}\times C_{2}\times C_{2}$ is not possible in  $S(32,4)$, $S(32,12)$ and $S(32,26)$. Hence $G'\cong S(32,2)$ or  $S(32,5)$ or $S(32,22)$ to $S(32,25)$ or $S(32,37)$ or $S(32,46)$ to $S(32,48)$.
 
If  $|\gamma_{5}(G)|=2$,, then $\gamma_{4}(G)\subseteq\zeta(G')$. Now if  $|\gamma_{4}(G)|=4$, then $|D_{(4),K}(G)|=8$ leads to $G'^{4}\gamma_{3}(G)^{2}\cap \gamma_{4}(G)=1$, but $\gamma_{4}(G)^{2}\neq1$. So $\gamma_{4}(G)\cong C_{4}\times C_{2}$ and $G'^{2}\subseteq\gamma_{3}(G)\cong C_{4}\times C_{2}\times C_{2}$. If $G'$ is abelian, then $G'\cong C_{8}\times C_{2}\times C_{2}$  or $C_{4}\times C_{4}\times C_{2}$ or $C_{4}\times C_{2}\times C_{2}\times C_{2}$. If $G'$ is nonabelian, then $\gamma_{4}(G)=\zeta(G')\cong C_{4}\times C_{2}$. Thus possible choices for $G'$ are $S(32,4)$ or $S(32,5)$ or $S(32,12)$ or  $S(32,24)$ or $S(32,25)$ or $S(32,26)$ or $S(32,37)$ or $S(32,48)$. Again we observe that $G'\ncong S(32,4)$, $S(32,12)$ and $S(32,26)$ as these groups do not contain a normal subgroup of the type $C_{4}\times C_{2}\times C_{2}$. Hence $G'\cong S(32,5)$ or  $S(32,24)$ or $S(32,25)$  or $S(32,37)$ or $S(32,48)$.\\~\\
{\bf(b)} $|\gamma_{3}(G)^{2}G^{\prime 4}|=4$. Then $\gamma_{5}(G)\subseteq \gamma_{3}(G)^{2}G^{\prime 4}\cong C_{2}\times C_{2}$. Also $\gamma_{5}(G)\subseteq\zeta(G')$ and $|\gamma_{4}(G)|=8$ or $4$. We consider these two cases separately :

 {\bf(i)} If $|\gamma_{4}(G)|=8$, then $G'^{2}\subseteq\gamma_{3}(G)$ and $\gamma_{3}(G)^{2}G^{\prime 4}\subseteq \gamma_{4}(G)\cong C_{4}\times C_{2}$. Also $\gamma_{4}(G)\subseteq\zeta(\gamma_{3}(G))$, so $\gamma_{3}(G)$ is  abelian. Since $\gamma_{3}(G)^{2}=\gamma_{3}(G)^{2}G^{\prime 4}\cong C_{2}\times C_{2}$, so $\gamma_{3}(G)\cong C_{4}\times C_{4}$. Thus if $G'$ is abelian, then $G'\cong C_{8}\times C_{4}$ or $C_{4}\times C_{4}\times C_{2}$. Let $G'$ be nonabelian. If $\gamma_{5}(G)=1$, then $\gamma_{3}(G)\subseteq\zeta(G')$ and $G'$ is abelian. So $\gamma_{5}(G)\neq1$.
 
If $|\gamma_{5}(G)|=4$, then $\gamma_{5}(G)=\gamma_{3}(G)^{2}G^{\prime 4}\cong C_{2}\times C_{2}$. Now if $\gamma_{6}(G)=1$, then $\gamma_{4}(G)\subseteq\zeta(G')$ and hence possible choices for $G'$ are $S(32,4)$ or $S(32,5)$ or $S(32,12)$ or $S(32,24)$ to $S(32,26)$  or $S(32,37)$ or $S(32,48)$. But a normal subgroup of the type $C_{4}\times C_{4}$ is not available in $S(32,5)$, $S(32,37)$ and $S(32,48)$, so these groups are not possible. Hence $G'\cong S(32,4)$ or $S(32,12)$ or $S(32,24)$ to $S(32,26)$. If $\gamma_{6}(G)\neq1$, then $C_{2}\times C_{2}\subseteq\zeta(G')$. So $\zeta(G')\cong C_{2}\times C_{2}$ or $C_{4}\times C_{2}$ or $C_{2}\times C_{2}\times C_{2}$. Therefore possible choices for $G'$ are $S(32,2)$ or $S(32,4)$ or $S(32,5)$ or $S(32,9)$ or $S(32,10)$ or $S(32,12)$ or $S(32,13)$ or $S(32,14)$ or $S(32,22)$ to  $S(32,35)$ or  $S(32,37)$ or $S(32,39)$ to $S(32,41)$ or $S(32,46)$ to  $S(32,48)$. But the groups $S(32,2)$, $S(32,5)$, $S(32,9)$, $S(32,10)$, $S(32,13)$, $S(32,14)$, $S(32,22)$, $S(32,23)$, $S(32,27) - S(32,30)$, $S(32,37)$, $S(32,39)-S(32,41)$ and $S(32,46)- S(32,48)$ do not have a normal subgroup of the type $ C_{4}\times C_{4}$. Hence $G'\cong S(32,4)$ or $S(32,12)$ or $S(32,24)$ to  $S(32,26)$ or  $S(32,31)$ to $S(32,35)$.

If $|\gamma_{5}(G)|=2$, then $\gamma_{4}(G)=\zeta(G')\cong C_{4}\times C_{2}$. So possible choices for $G'$ are  $S(32,4)$ or $S(32,5)$ or $S(32,12)$ or $S(32,24)$ to $S(32,26)$ or $S(32,37)$ or $S(32,48)$. But $\gamma_{3}(G)\cong C_{4}\times C_{4}$ is not possible in $S(32,5)$, $S(32,37)$, $S(32,38)$ and $S(32,48)$. Hence $G'\cong S(32,4)$ or $S(32,12)$ or $S(32,24)$ or $S(32,25)$ or $S(32,26)$. 

{\bf(ii)} If $|\gamma_{4}(G)|=4$, then $\gamma_{4}(G)\cong C_{4}$, $|\gamma_{3}(G)^{2}G^{\prime 4}\cap \gamma_{4}(G)|=2$, $\gamma_{4}(G)\subseteq\zeta(G')$ and $|\gamma_{3}(G)|=16$ or $8$.

If $|\gamma_{3}(G)|=16$, then $G'^{2}\subseteq\gamma_{3}(G)$. Therefore $\gamma_{3}(G)^{2}\cong C_{2}\times C_{2}$ and hence $\gamma_{3}(G)\cong C_{4}\times C_{4}$. If $G'$ is abelian, then $G'\cong C_{8}\times C_{4}$ or  $C_{4}\times C_{4}\times C_{2}$. If $G'$ is nonabelian, then $\gamma_{5}(G)\cong C_{2}$, $cl(G')= 2$ and $\zeta(G')\cong C_{4}$ or $C_{8}$ or $C_{4}\times C_{2}$. So possible choices for $G'$ are $S(32,4)$ or $S(32,5)$ or $S(32,12)$ or $S(32,24)$ or $S(32,25)$ or $S(32,26)$ or $S(32,37)$ or  $S(32,38)$ or $S(32,48)$. Since a normal subgroup of the type $C_{4}\times C_{4}$ is not available in $S(32,5)$, $S(32,37)$, $S(32,38)$ and $S(32,48)$, so these groups are not possible. Hence $G'\cong S(32,4)$ or $S(32,12)$ or $S(32,24)$ or $S(32,25)$ or $S(32,26)$.

Let $|\gamma_{3}(G)|=8$. Then $\gamma_{3}(G)\cong C_{4}\times C_{2}$. If $|G'^{2}|=2$, then $G'^{2}\cap \gamma_{3}(G)\neq 1$, as $|\gamma_{3}(G)^{2}G^{\prime 4}\cap \gamma_{4}(G)|=2$. So $|G'^{2}|=8$ or $4$. If $|G'^{2}|=8$, then $|G'^{2}\cap \gamma_{3}(G)|= 4$. Suppose that $G'$ is abelian. Then $G'\cong C_{8}\times C_{4}$. But then $\gamma_{3}(G)\subseteq G^{\prime 2}$, which is not possible. So  $G'$ is nonabelian. If $\gamma_{5}(G)=1$, then $\gamma_{3}(G)=\zeta(G')\cong C_{4}\times C_{2}$. So the possible choices for $G'$ are $S(32,4)$ or $S(32,5)$ or $S(32,12)$. But in all these cases $G'^{2}=\zeta(G')=\gamma_{3}(G)$. If $\gamma_{5}(G)\cong C_{2}$, then $\zeta(G')\cong C_{4}$ or $C_{8}$ or $C_{4}\times C_{2}$. Hence $G'\cong S(32,4)$ or $S(32,5)$ or $S(32,12)$ or $S(32,17)$. If $|G'^{2}|=4$, then $|G'^{2}\cap \gamma_{3}(G)|=2$. Suppose that $G'$ is abelian. Then $G'\cong C_{8}\times C_{2}\times C_{2}$ or $C_{4}\times C_{4}\times C_{2}$. Let $G'$ be nonabelian. If $\gamma_{5}(G)=1$, then $\gamma_{3}(G)=\zeta(G')\cong C_{4}\times C_{2}$. So the possible choices for $G'$ are $S(32,24)$ or $S(32,25)$ or $S(32,26)$ or $S(32,37)$. But in all these cases $G'^{2}\subseteq\zeta(G')=\gamma_{3}(G)$. If $\gamma_{5}(G)\cong C_{2}$, then $\zeta(G')\cong C_{4}$ or $C_{8}$ or $C_{4}\times C_{2}$. Hence $G'\cong S(32,24)$ or $S(32,25)$ or $S(32,26)$ or $S(32,37)$ or $S(32,38)$.

Conversely, if conditions (i) to (v) hold, then $d_{(2)}=2$, $d_{(3)} = d_{(5)} = d_{(9)}=1$ and for the remaining conditions $d_{(2)} = d_{(3)}= d_{(4)} = d_{(5)} = d_{(7)} = 1$. Thus $n=5$ and $t^L(KG)=18=2^n-14$.

\end{proof}
\begin{lemma}{\label{l4}}
	Let $K$ be a field of characteristic $p=17$ and let $G$ be a nilpotent group such that $\vert G'\vert = 17^{n}$. Then $t^{L}(KG)= 17^{n}-239$ if and only if $G'\cong C_{17}\times C_{17}$, $\gamma_{3}(G)\cong C_{17}$
\end{lemma}
\begin{proof}
Let $t^{L}(KG)= 17^{n}-239$, then by Lemma \ref{l2},  $d_{(2)} = d_{(3)}=1$. Thus $|G'|= 17^{2}$, $|D_{(3),K}(G)|=17$ . So  $t^{L}(KG)= 50 = 3p-1$. Rest follows by \cite[Theorem 3.5]{Sa}.
\end{proof}
\begin{theorem}
	Let $K$ be a field of characteristic $p>0$ and let $G$ be a nilpotent group such that $\vert G'\vert = p^{n}$. Then $t^{L}(KG)= |G'|- 15p+ 16$ if and only if one of the following conditions holds:

\begin{enumerate}
\item $p=2$ and $G'$ is a group of one of the types in Lemma~\ref{l3};

\item $p=17$, $G'\cong C_{17}\times C_{17}$, $\gamma_{3}(G)\cong C_{17}$
\end{enumerate}
\end{theorem}
\begin{proof} Lemma~\ref{l3} and Lemma~\ref{l4} together imply the complete proof of the Theorem.
	
\end{proof}

\begin{table}[htp]
\caption{ Nonabelian Groups of order $2^{5}$}
\label{tab:table1}	  
\begin{tabular}{|c|c|c|c|c|c|c|c|} 
\hline		
$G$ & $\exp(G)$ & $\zeta(G)$  & $G^{2}$ & $G^{4}$  & $G^{2}\cap \zeta(G)$ & $G^{4}\cap \zeta(G)$   & $cl(G)$ \\
\hline	$S(32, 2)$ & 4 & $(C_{2})^{3}$  & $(C_{2})^{3}$ & 1 & $ (C_{2})^{3}$ & 1  & 2 \\
\hline	$S(32, 4)$ & 8 & $C_{4}\times C_{2}$  & $C_{4}\times C_{2}$ & $C_{2}$ &  $C_{4}\times C_{2}$ & $C_{2}$ & 2 \\
\hline	$S(32, 5)$ & 8 & $C_{4}\times C_{2}$  & $C_{4}\times C_{2}$ & $C_{2}$ &  $C_{4}\times C_{2}$ & $C_{2}$ & 2 \\
\hline	$S(32, 6)$ & 4 & $C_{2}$  & $ (C_{2})^{3}$ & 1 &  $C_{2}$ & 1   & 3 \\
\hline	$S(32, 7)$ & 8 & $C_{2}$  & $C_{4}\times C_{2}$ & $C_{2}$ &  $C_{2}$ & $C_{2}$  & 3 \\
\hline	$S(32, 8)$ & 8 & $C_{2}$  & $C_{4}\times C_{2}$ & $C_{2}$ &  $C_{2}$ & $C_{2}$  & 3 \\
\hline	$S(32, 9)$ & 8 & $C_{2}\times C_{2}$  & $C_{4}\times C_{2}$ & $C_{2}$ &  $C_{2}\times C_{2}$ & $C_{2}$   & 3 \\
\hline	$S(32, 10)$ & 8 & $C_{2}\times C_{2}$  & $C_{4}\times C_{2}$ & $C_{2}$ &  $C_{2}\times C_{2}$ & $C_{2}$   & 3 \\
\hline	$S(32, 11)$ & 8 & $C_{4}$  & $C_{4}\times C_{2}$ & $C_{2}$ &  $C_{4}$ & $C_{2}$  & 3 \\
\hline	$S(32, 12)$ & 8 & $C_{4}\times C_{2}$ & $C_{4}\times C_{2}$ & $C_{2}$ &  $C_{4}\times C_{2}$ & $C_{2}$   & 2 \\
\hline	$S(32, 13)$ & 8 & $C_{2}\times C_{2}$  & $C_{4}\times C_{2}$ & $C_{2}$ &  $C_{2}\times C_{2}$ & $C_{2}$  & 3 \\
\hline	$S(32, 14)$ & 8 & $C_{2}\times C_{2}$  & $C_{4}\times C_{2}$ & $C_{2}$ & $C_{2}\times C_{2}$ & $C_{2}$  & 3 \\
\hline	$S(32, 15)$ & 8 & $C_{4}$  & $C_{4}\times C_{2}$ &  1& $C_{4}$ & $C_{2}$   & 3 \\
\hline	$S(32, 17)$ & 16 & $C_{8}$  & $C_{8}$ & $C_{4}$ &  $C_{8}$ & $C_{4}$ & 2 \\
\hline	$S(32, 18)$ & 16 & $C_{2}$  & $C_{8}$ & $C_{4}$ & $C_{2}$ & $C_{2}$   & 4 \\
\hline	$S(32, 19)$ & 16 & $C_{2}$ & $C_{8}$ & $C_{4}$  & $C_{2}$ & $C_{2}$ &  4 \\
\hline	$S(32, 20)$ & 16 & $C_{2}$ & $C_{8}$ & $C_{4}$  & $C_{2}$ & $C_{2}$  & 4 \\
\hline	$S(32, 22)$ & 4 & $(C_{2})^{3}$  & $C_{2}\times C_{2}$ & 1 &  $C_{2}\times C_{2}$ & 1 &  2 \\
\hline	$S(32, 23)$ & 4 & $(C_{2})^{3}$  & $C_{2}\times C_{2}$ & 1  & $C_{2}\times C_{2}$ & 1 &  2 \\
\hline	$S(32, 24)$ & 4 & $C_{4}\times C_{2}$  & $C_{2}\times C_{2}$ & 1 &  $C_{2}\times C_{2}$ & 1 &  2 \\
\hline	$S(32, 25)$ & 4 & $C_{4}\times C_{2}$ & $C_{2}\times C_{2}$ & 1 &  $C_{2}\times C_{2}$ & 1 & 2 \\
\hline	$S(32, 26)$ & 4 & $C_{4}\times C_{2}$ & $C_{2}\times C_{2}$ & 1  & $C_{2}\times C_{2}$ & 1 & 2 \\
\hline	$S(32, 27)$ & 4 & $C_{2}\times C_{2}$ & $C_{2}\times C_{2}$ & 1 &  $C_{2}\times C_{2}$ & 1 &  2 \\
\hline	$S(32, 28)$ & 4 & $C_{2}\times C_{2}$ & $C_{2}\times C_{2}$ & 1 &  $C_{2}\times C_{2}$ & 1 & 2 \\
\hline	$S(32, 29)$ & 4 & $C_{2}\times C_{2}$  & $C_{2}\times C_{2}$ & 1 &  $C_{2}\times C_{2}$ & 1 &  2 \\
\hline	$S(32, 30)$ & 4 & $C_{2}\times C_{2}$ & $C_{2}\times C_{2}$ & 1 &  $C_{2}\times C_{2}$ & 1 &  2 \\
\hline	$S(32, 31)$ & 4 & $C_{2}\times C_{2}$ & $C_{2}\times C_{2}$ & 1  & $C_{2}\times C_{2}$ & 1 &  2 \\
\hline	$S(32, 32)$ & 4 & $C_{2}\times C_{2}$ & $C_{2}\times C_{2}$ & 1 &  $C_{2}\times C_{2}$ & 1 &  2 \\
\hline	$S(32, 33)$ & 4 & $C_{2}\times C_{2}$ & $C_{2}\times C_{2}$ & 1 &  $C_{2}\times C_{2}$ & 1 & 2 \\
\hline	$S(32, 34)$ & 4 & $C_{2}\times C_{2}$  & $C_{2}\times C_{2}$ & 1 &  $C_{2}\times C_{2}$ & 1 &  2 \\
\hline	$S(32, 35)$ & 4 & $C_{2}\times C_{2}$  & $C_{2}\times C_{2}$ & 1 &  $C_{2}\times C_{2}$ & 1 &  2 \\
		
\hline	$S(32, 37)$ & 8 & $C_{4}\times C_{2}$  & $C_{4}$ & $C_{2}$  & $C_{4}$ & $C_{2}$ & 2 \\
		
\hline	$S(32, 38)$ & 8 & $C_{8}$ & $C_{4}$ & $C_{2}$  & $C_{4}$ & $C_{2}$ &  2 \\
		
\hline	$S(32, 39)$ & 8 & $C_{2}\times C_{2}$ & $C_{4}$ & $C_{2}$ & $C_{2}$ & $C_{2}$ & 3 \\
		
\hline	$S(32, 40)$ & 8 & $C_{2}\times C_{2}$ & $C_{4}$ & $C_{2}$  & $C_{2}$ & $C_{2}$ & 3 \\
		
\hline	$S(32, 41)$ & 8 & $C_{2}\times C_{2}$ & $C_{4}$ & $C_{2}$  & $C_{2}$ & $C_{2}$ &  3 \\
		
\hline	$S(32, 42)$ & 8 & $C_{4}$ & $C_{4}$ & $C_{2}$ &  $C_{4}$ & $C_{2}$ & 3 \\
		
\hline	$S(32, 43)$ & 8 & $C_{2}$ & $C_{4}$ & $C_{2}$ &  $C_{2}$ & $C_{2}$ &  3 \\
		
\hline	$S(32, 44)$ & 8 & $C_{2}$ & $C_{4}$ & $C_{2}$ &  $C_{2}$ & $C_{2}$ & 3 \\
		
\hline	$S(32, 46)$ & 4 & $(C_{2})^{3}$ & $C_{2}$ & 1 &  $ C_{2}$ & 1 & 2 \\
		
\hline	$S(32, 47)$ & 4 & $(C_{2})^{3}$ & $C_{2}$ & 1 &  $ C_{2}$ & 1 & 2 \\
		
\hline	$S(32, 48)$ & 4 & $C_{4}\times C_{2}$ & $C_{2}$ & 1 & $ C_{2}$ & 1 & 2 \\
		
\hline	$S(32, 49)$ & 4 & $C_{2}$ & $C_{2}$ & 1 &  $ C_{2}$ & 1 & 2 \\
		
\hline	$S(32, 50)$ & 4 & $C_{2}$ & $C_{2}$ & 1 &  $ C_{2}$ & 1 & 2 \\
\hline	   
\end{tabular}
\end{table}

{\bf{Acknowledgments}}: The financial assistance provided to the second author in the form of a Senior Research Fellowship from University Grants Commission, India is gratefully acknowledged.

\end{document}